\documentclass[11pt]{amsart}
\usepackage[totalwidth=440pt, totalheight=640pt]{geometry}
\usepackage{amsmath, amssymb, amsthm, amsfonts, mathrsfs, eucal, enumerate, natbib, layout}
\usepackage[normalem]{ulem}
\usepackage[all,tips,2cell]{xy}
\usepackage[usenames]{color}
\usepackage{verbatim}
\usepackage{mathtools}

\newcommand{\details}[1]{}

\newtheorem{theorem}{Theorem}[section]
\newtheorem*{theorem*}{Theorem}

\newtheorem*{corollary*}{Corollary}
\newtheorem{lemma}[theorem]{Lemma}

\newtheorem*{claim*}{Claim}
\newtheorem*{lemma*}{Lemma}
\newtheorem{proposition}[theorem]{Proposition}
\newtheorem*{proposition*}{Proposition}

\newtheorem*{conjecture*}{Conjecture}
\newtheorem{def-proposition}[theorem]{Definition-Proposition}
\theoremstyle{definition}

\newtheorem*{definition*}{Definition}

\newtheorem*{example*}{Example}
\numberwithin{equation}{section}

\newcommand{\GG}{\mathbb{G}}

\newcommand{\Gr}{\mathrm{Gr}}

\newcommand{\HH}{\mathrm{H}}
\newcommand{\R}{\mathrm{R}}
\newcommand{\can}{\mathrm{can}}
\newcommand{\pr}{\mathrm{pr}}
\newcommand{\Pic}{\mathrm{Pic}}

\newcommand{\cL}{\mathcal{L}}
\newcommand{\cO}{\mathcal{O}}

\newcommand{\Corr}{\mathrm{\textbf{Corr}}}
\newcommand{\uPic}{\mathrm{\textbf{Pic}}}
\newcommand{\uHom}{\mathrm{\textbf{Hom}}}

\newcommand{\spec}{{\mathrm{Spec}}\,}
\newcommand{\ok}{\overline{k}}

\begin{document}

\title[Divisorial correspondences]
{A note on divisorial correspondences of \\ extensions of abelian schemes by tori}

\author{Cristiana Bertolin}
\address{Dipartimento di Matematica, Universit\`a di Torino, Via Carlo Alberto 10, Italy}
\email{cristiana.bertolin@unito.it}
\author{Federica Galluzzi}
\address{Dipartimento di Matematica, Universit\`a di Torino, Via Carlo Alberto 10, Italy}
\email{federica.galluzzi@unito.it}



\dedicatory{Dedicated to M. Raynaud}



\begin{abstract} Let $S$ be a locally noetherian scheme and consider two extensions $G_1$ and $G_2$ of abelian $S$-schemes by $S$-tori. In this note we prove that
the $fppf$-sheaf $\Corr_S(G_1,G_2)$ of divisorial correspondences between $G_1$ and $G_2$ is representable. Moreover, using divisorial correspondences, we show that line bundles on an extension $G$ of an abelian scheme by a torus define group homomorphisms between $G$ and $ \uPic_{ G/S}.$ 
\end{abstract}


\maketitle


\tableofcontents

\section*{Introduction}

In algebraic geometry, the notion of correspondences between varieties plays an important role for the study of algebraic cycles and motives. In this short note we discuss a special case: divisorial correspondences between
 group schemes which are extensions of an abelian
scheme by a torus over a fixed locally Noetherian base $S$.
Let $G$ be such an extension of an abelian $S$-scheme by an $S$-torus.
Denote by $p:G \to S$ its structural morphism.
 The \textbf{relative Picard functor} associated to $G/S$ is the $fppf$-sheaf $\uPic _{G/S} = \R^1p_{*} \GG _m$, i.e. the $fppf$-sheaf associated to the presheaf $T/S \mapsto \Pic(G_T),$ where $\Pic(G_T)$ is the group of isomorphism classes of invertible sheaves on the $T$-scheme $G_T= G \times_S T$ obtained from $G$ by the base change $ T \to S.$

The $S$-group scheme $G$ admits a unit section $\epsilon: S \to G$. From now on, we will assume that the structural morphism $p:G \to S$ satisfies $p_{*}\mathcal O_{G}=\mathcal O_{S}$ universally.  With these hypotheses
 the $fppf$-sheaf $\uPic _{G/S}$ is canonically isomorphic to the \'etale-sheaf $ \R^1p_{*} \GG_m $ and moreover it is canonically isomorphic to the sheaf 
 $\uPic_{G/S}^\epsilon : T/S \mapsto \Pic(G_T) / \Pic (T),$
where $\Pic(G_T) / \Pic (T)$ is the group of isomorphism classes of invertible sheaves $\cL$ on $ G_T$ which are rigidified along the unit section $\epsilon_T : T \to G_T $ obtained from $\epsilon: S \to G$ by the base change $T \to S$, that is it exists an isomorphism between the structural sheaf $\cO_T$ and $\epsilon_T^* \cL.$ We call this isomorphism $\cO_T \cong \epsilon_T^* \cL $ a rigidification of $\cL$ along $\epsilon_T$.

Now consider two extensions $p_1:G_1 \to S$, $p_2:G_2 \to S$ of abelian schemes by tori. Denote by $\epsilon_i: S \to G_i$ their unit sections
 and suppose that $p_i$ satisfy $p_{i_*}\mathcal O_{G_i}=\mathcal O_{S}$ universally for $i=1,2$.  
Consider the canonical morphism of sheaves defined by pull-backs

\begin{equation}\label{cube}
\begin{matrix}
\can: \, \uPic_{ G_1 /S} \times \uPic_{ G_2/S}  & \longrightarrow &  
\uPic_{ G_1 \times _S G_2/S} \\
(\cL_1,\cL_2) &  \longmapsto & pr_1^* \cL_1 \otimes pr_2^* \cL_2
\end{matrix}
\end{equation}
where $\pr_i: G_1 \times_S G_2 \to G_1$ are the projections to the $i$-th factor for $i=1,2$.  The \textbf{sheaf of divisorial correspondences between $G_1$ and $G_2$ over $S$}, that we denote by
$$\Corr _S(G_1,G_2),$$
 is the $fppf$-sheaf cokernel of $\can$ (\ref{cube}). We have an exact sequence of $fppf$-sheaves
 \[\uPic_{ G_1 /S} \times \uPic_{ G_2/S}  \stackrel{\can}{\longrightarrow} \uPic_{ G_1 \times _S G_2/S} \longrightarrow  \Corr_S(G_1,G_2) \longrightarrow 0. \]
 Since the extensions $G_i$ are endowed with the unit sections $\epsilon_i$ and since we have supposed $p_{i_*}\mathcal O_{G_i}=\mathcal O_{S}$ universally, 
 using the rigidified version $\uPic_{G_i/S}^{\epsilon_i}$ of the relative Picard functor 
 we get that for any $S$-scheme $T$ the sequence
 \begin{equation}\label{eq:corr=pic2fattori}
 0 \longrightarrow \uPic_{ G_1 /S}(T) \times \uPic_{ G_2/S}(T) \stackrel{\can}{\longrightarrow} \uPic_{ G_1 \times _S G_2/S}(T) \longrightarrow  \Corr_S(G_1,G_2)(T) \longrightarrow 0,
 \end{equation}
 is exact, that is  $\Corr _S(G_1,G_2)(T)$ is the group of isomorphism classes of the invertible sheaves on $G_{1T}\times_T G_{2T} $ endowed with rigidifications along $\epsilon _1 \times _T G_{2T}$ and along $G_{1T} \times _T \epsilon _{2T} $ which must agree on $\epsilon _{1T} \times _T \epsilon _{2T}$.

 The aim of this note is to prove that the $fppf$-sheaf $\Corr _S(G_1,G_2)$ of divisorial correspondences between $G_1$ and $G_2$ is representable (Theorem \ref{thm:RepCorr}). Moreover, using divisorial correspondences, we show that line bundles on an extension $G$ of an abelian scheme by a torus define group homomorphisms between $G$ and $ \uPic_{ G/S}$ (Proposition \ref{prop}).
 In \cite[Thm 0.1, Thm 5.1]{BB} S. Brochard and the first author construct the morphism defined in (\ref{equ:linebund=morph}) for 1-motives without using divisorial correspondences and they prove the Theorem of the Cube for 1-motives. In \cite[Thm 5.9.]{BG} the authors prove the generalized Theorem of the Cube for 1-motives.

This paper takes the origin from an exchange of emails with M. Raynaud.
We want to thank M. Brion for his comments about the hypothesis ``$p_* \mathcal{O}_G=\mathcal{O}_S$ universally'', we use in this paper.

\section{Representability of $\Corr$}

In \cite[Thm 1]{Murre} Murre gives a criterion for a contravariant functor from the category of schemes over $S$ to the category of sets to be representable by an unramified, separated $S$-scheme which is locally of finite type over $S$. Using this criterion, he proves the representability of the $fppf$-sheaf $\Corr _S(X_1,X_2)$ with $ X_1 $ and $X_2$ proper and flat $S$-schemes 
(see \cite[Thm 4]{Murre}). We adapt his results to extensions of abelian schemes by tori which are not proper.

\begin{theorem}\label{thm:RepCorr}
Consider two extensions $p_1:G_1 \to S$, $p_2:G_2 \to S$ of abelian schemes by tori. Suppose that the structural morphisms $p_i$ satisfy $p_{i_*}\mathcal O_{G_i}=\mathcal O_{S}$ universally for $i=1,2$.  
The $fppf$-sheaf $\mathbf{Corr}_S(G_1,G_2)$ of divisorial correspondences between $G_1$ and $G_2$ is representable by an $S$-group scheme, locally of finite presentation, separated and unramified over $S$.
\end{theorem} 

\begin{proof}
We have to prove that the functor $\Corr _S(G_1,G_2)$  verifies the properties $(F_1),...,(F_8)$  listed in \cite[Thm 1]{Murre}. Since the structural morphisms $p_1,p_2$ have sections, 
$(F_1),(F_2),(F_4)$ follow from the same properties of $\uPic_{ G_1 \times _S G_2/S}$. 
Concerning property $(F3),$ by \cite[Prop II 2.4 (2) (i)]{Raynaud70} the extension $G_i$ (for $i=1,2$) is $S$-pure
and therefore \cite[Chp 37, Lem 27.6 (2), Def 21.1]{StPr} implies that there exists a universal flattening of $G_i$, that is the flattening functor is representable. Now using \cite[Thm 2]{Murre}, $(F_3)$ follows from the same property of $\uPic_{ G_1 \times _S G_2/S}$.
Property $(F_5)$  (i.e. the fact that $\Corr _S(G_1,G_2)$ is formally unramified) follows by \cite[Prop III 4.1]{Raynaud70}. Property $(F_6)$  (i.e. the fact that $\Corr _S(G_1,G_2)$ is separated) follows by \cite[Prop III 4.3]{Raynaud70}. For Properties $(F_7),(F_8)$ see Murre's proof in \cite[Thm 4]{Murre}.
\end{proof}

The assumption 
\begin{equation}\label{eq:1}
p_{*}\mathcal O_{G}=\mathcal O_{S} \qquad  \mathrm{universally}
\end{equation}
 is not too restrictive. For example, anti-affine algebraic groups over a field $k$, which is not an algebraic extension of a finite field, furnish \textit{non-trivial} examples where the condition (\ref{eq:1}) holds.
More precisely, we have the following geometrical interpretation of the condition (\ref{eq:1}):

\begin{lemma}
Let $k$ be a field and let $\ok$ its algebraic closure.
 Consider an extension $G$ of an abelian variety $A$ by a torus $T$ defined over a field $k$. Let $p:G \to S= \spec (k)$ be the structural morphism of $G$. Denote by $c: X^*(T)(\ok) \to A^* (\ok)$ the $\mathrm{Gal} (\ok/k)$-equivariant homomorphism which defines the extension $G$, where $A^*$ is the dual abelian variety of $A$ and $X^*(T)$ is the character group of the torus $T$. Then the following conditions are equivalent:
\begin{enumerate}
	\item the structural sheaf of the extension $G$ satisfies the condition (\ref{eq:1}),
	\item the $\mathrm{Gal} (\ok/k)$-equivariant homomorphism $c: X^*(T)(\ok) \to A^* (\ok)$ is injective,
	\item the extension $G$ is anti-affine, that is $\mathcal O_{G} (G)=k.$
\end{enumerate}
\end{lemma}

\begin{proof} The equivalence between (2) and (3) is given by \cite[Prop 2.1]{Br09}.
If (1) holds, we have that $\mathcal O_{G}(G)=p_{*}\mathcal O_{G} (\spec (k) ) =\mathcal O_{S}(\spec (k)) =k$, i.e. the extension $G$ is anti-affine.
Suppose now that (3) holds. Denote by $\alpha: G \to A$ the surjective morphism of algebraic groups underlying the extension $G$ and by $q: A \to S$ the structural morphism of $A$. 
Let $G_{\ok}$ the extension obtained from $G$ extending the scalars from $k$ to $\ok$.
As observed in \cite[(2.2)]{Br09} $\alpha_*(\mathcal O_{G_{\ok}})= \bigoplus_{x \in X^*(T)} {\cL}_x$ where ${\cL}_x$ is the invertible sheaf on $A_{\ok}$ algebraically equivalent to 0, which corresponds to the  point $c(x)$ of $A_{\ok}^*(\ok)$ via the isomorphism $A^* \cong 
\uPic_{ A/S}^0$. Hence 
\[ p_{*}\mathcal O_{G_{\ok}} = q_* \;  \alpha_* \mathcal O_{G_{\ok}} = q_* \bigoplus_{x \in X^*(T)} {\cL}_x .\]
Since $G$ is anti-affine, 
$\HH^0(A_{\ok}, {\cL}_x )=0$ for all $x \not= 0$. Therefore
\[ p_{*}\mathcal O_{G_{\ok}} ( \spec (k) )
= \bigoplus_{x \in X^*(T)} \HH^0(A_{\ok}, {\cL}_x  ) = \HH^0(A_{\ok}, {\cL}_0  ) = \mathcal O_{A_{\ok}} ( A_{\ok} ) =k,\]
that is $ p_{*}\mathcal O_{G_{\ok}} = \mathcal O_{S}.$
\end{proof}

Consider an extension $G$ of an abelian variety $A$ by a torus $T$ defined over a field $k$. As before denote by $c: X^*(T)(\ok) \to A^* (\ok)$ the $\mathrm{Gal} (\ok/k)$-equivariant homomorphism which defines this extension $G$. Let $X''$ be the biggest $\mathrm{Gal} (\ok/k)$-sub-module of $X^*(T)(\ok)$ whose image via $c: X^*(T)(\ok) \to A^* (\ok)$ is a torsion subgoup of $A^* (\ok).$ Denote by $T''$ the quotient torus of $T$ whose character group is $X''$. Then
 $G$ is an extension of the torus $T''$ by an extension $G'$ of A by $T/T''$
\[ 0 \longrightarrow G' \longrightarrow G \longrightarrow T'' \longrightarrow 0.\]
Now the $\mathrm{Gal} (\ok/k)$-equivariant homomorphism $X^*(T/T'')(\ok) \to A^* (\ok)$ defining the extension $G'$ is injective, and therefore, by the above Lemma, $G'$ is anti-affine, 
that is the global functions of $G'$ are $k$, or equivalently for $G'$ the condition (\ref{eq:1}) holds. 
Since the torus $T''$ plays no role for the study of divisorial correspondences, we have showed that over a field we can always reduce to the case where condition (\ref{eq:1}) holds.

\section{Linear morphisms via divisorial correspondences}

Consider two extensions $p_1:G_1 \to S$, $p_2:G_2 \to S$ of abelian schemes by tori.  Denote by $\epsilon_i: S \to G_i$ their unit sections and suppose that the structural morphisms $p_i$ satisfy $p_{i_*}\mathcal O_{G_i}=\mathcal O_{S}$ universally for $i=1,2$.
Let $\uHom_{\epsilon_1} (G_1, \uPic_{G_2/S})$ be the sheaf of morphisms of sheaves from $G_1$ to $\uPic_{G_2/S}$ which send the unit section $\epsilon_1$ to the unit section of $\uPic_{G_2/S}$  and likewise for $\uHom_{\epsilon_2} (G_2, \uPic_{G_1/S})$.
Observe that for any $S$-scheme $T$, $\uHom_{\epsilon_1} (G_1, \uPic_{G_2/S})(T)$ is just the group $\uPic_{G_{2T}/S}(G_{1T}) $ of $G_{1T}$-points of $\uPic_{G_{2T}/S}$. Then by the short exact sequence (\ref{eq:corr=pic2fattori}) we have that 
 the $fppf$-sheaf $T/S \mapsto \uHom_{\epsilon_{1T}} (G_{1T}, \uPic_{G_{2T}/S})$ is isomorphic to the $fppf$-sheaf $T/S \mapsto \Corr _S(G_1,G_2)(T).$
 Therefore we have the isomorphisms of $fppf$-sheaves
 \begin{equation} \label{eq:corr=arrow}
 \uHom_{\epsilon_1} (G_1, \uPic_{G_2/S}) \cong \Corr _S(G_1,G_2) \cong \uHom_{\epsilon_2} (G_2, \uPic_{G_1/S})
\end{equation}

We define by $\mathbf{Hom}_{\Gr}(X,Y)$ the sheaf of group homorphisms between abelian sheaves.

\begin{proposition}\label{prop}
Let $S$ be a normal scheme. Consider two extensions $p_1:G_1 \to S$, $p_2:G_2 \to S$ of abelian schemes by tori.  Suppose that the structural morphisms $p_i$ satisfy $p_{i_*}\mathcal O_{G_i}=\mathcal O_{S}$ universally for $i=1,2$. Then 
\begin{equation}\label{equ:homo}
\mathbf{Hom}_{\Gr}(G_1,\mathbf{Pic}_{G_2/S}) \cong \mathbf{Corr} _S(G_1,G_2) \cong \mathbf{Hom}_{\Gr}(G_2,\mathbf{Pic}_{G_1/S}).
\end{equation}
In particular, if $G$ is an extension of an abelian scheme by a torus over a normal base scheme, this yields a morphism of $fppf$-sheaves
\begin{equation}\label{equ:linebund=morph} 
\mathbf{Pic}_{G/S} \longrightarrow \mathbf{Hom}_{\Gr}(G,\mathbf{Pic}_{G/S}).
\end{equation}
\end{proposition}

\begin{proof} Because of (\ref{eq:corr=arrow}), it is enough to show that if $u: G_1 \to \uPic _{G_2/S}$ is a morphism of sheaves which sends the unit section  $\epsilon_1: S \to G_1$ of $G_1$  to the unit section of $\uPic _{G_2/S}$, then $u$ is in fact a group homomorphism. We will prove that the following morphism of sheaves
\[
\begin{matrix}
v: \,  G_1 \times _S  G_1  & \longrightarrow &  
\uPic_{G_2 /S} \\
(g,g') &  \longmapsto & u(g+g') -u(g) - u(g')
\end{matrix}
\]
is the null morphism. The morphism $v$ is a $G_1 \times _S G_1$-point of $\uPic_{G_2/S}$, that is an invertible sheaf $\cL$ on $G_1 \times _S G_1 \times _S G_2$ that we can suppose to be rigidified along the unit section $\epsilon_2 : S \to G_2$ of $G_2$. Since by hypothesis $u(0)=0,$ the restriction of $\cL$ to $G_1 \times _S S \times _S  G_2$ 
 and to $S \times _S G_1 \times _S  G_2$ is trivial. Therefore $\cL$ is rigidified along $\epsilon _1 \times _S G_1 \times _S G_2 $ , $G_1 \times _S \epsilon _1 \times _S G_2 $ and $G_1 \times _S G_1 \times _S \epsilon _2$. But by \cite[Chp I, \S 2.6]{MB85}, the extensions $G_i$ over a normal base scheme satisfy 
the Theorem of the Cube, and so the line bundle $\cL$ is trivial, that is $v$ is the null morphism.

Now let $G$ be an extension of an abelian scheme by a torus. Consider the canonical morphisms of $fppf$-sheaves
\[
\uPic  _{G/S} \to  \uPic_{G/S} \times _S \uPic_{G/S} \stackrel{\can}{\longrightarrow} \uPic_{G \times _S G}  \to \Corr _S (G,G). 
\]
Using (\ref{equ:homo}) we get the expected morphism of $fppf$-sheaves. 
\end{proof}

\bibliographystyle{plain}

\end{document}